\documentclass[conference]{IEEEtran}
\IEEEoverridecommandlockouts
\usepackage{cite}
\usepackage{amsmath,amssymb,amsfonts,amsthm}
\usepackage{hyperref}
\usepackage{graphicx}
\graphicspath{{figures/}}
\usepackage{textcomp}
\usepackage[table]{xcolor}
\def\BibTeX{{\rm B\kern-.05em{\sc i\kern-.025em b}\kern-.08em
    T\kern-.1667em\lower.7ex\hbox{E}\kern-.125emX}}

    \newcommand\textopenone{\leavevmode\hbox{\small 1\kern-3.3pt\normalsize 1}}

\usepackage{algorithm}
\usepackage{algpseudocode}
\usepackage{subcaption}
\usepackage{pifont}   
\usepackage{comment}
\usepackage{lineno}
\usepackage[capitalise]{cleveref}

\newtheorem{theorem}{Theorem}

\newcommand{\LineComment}[1]{\Statex \hfill $\triangleright$ \textit{#1}}

\newcommand{\nn}{n}
\newcommand{\nm}{m}

\newcommand{\real}{\mathbb{R}}

\newcommand{\rnn}{{\real}^{\nn \times \nn}}
\newcommand{\rmn}{{\real}^{\nm \times \nn}}

\newcommand\inv{^{-1}}

\newcommand{\ma}{\boldsymbol{A}}

\newcommand{\md}{\boldsymbol{D}}
\newcommand{\me}{\boldsymbol{E}}
\newcommand{\mf}{\boldsymbol{F}}
\newcommand{\mg}{\boldsymbol{G}}

\newcommand{\mi}{\boldsymbol{I}}

\newcommand{\mq}{\boldsymbol{Q}}
\newcommand{\mr}{\boldsymbol{R}}
\newcommand{\ms}{\boldsymbol{S}}

\newcommand{\mU}{\boldsymbol{U}}
\newcommand{\mv}{\boldsymbol{V}}

\newcommand{\vu}{\mathbf{u}}

\DeclareMathOperator{\rank}{\mathrm{rank}}

\begin{document}
\setlength{\linenumbersep}{3pt}

\title{Numerically Stable Cholesky-QR on GPU via Mixed-Precision Randomized Preconditioning
}

\author{
 \IEEEauthorblockN{James E. Garrison}
\IEEEauthorblockA{\textit{Dept. of Mathematics} \\
\textit{North Carolina State University}\\
Raleigh, NC, USA \\
jegarri3@ncsu.edu}
\and
\IEEEauthorblockN{Chao Chen}
\IEEEauthorblockA{\textit{Dept. of Mathematics} \\
\textit{North Carolina State University}\\
Raleigh, NC, USA \\
cchen49@ncsu.edu}
\and
\IEEEauthorblockN{Ilse C. F. Ipsen}
\IEEEauthorblockA{\textit{Dept. of Mathematics} \\
\textit{North Carolina State University}\\
Raleigh, NC, USA \\
ipsen@ncsu.edu}
}

\maketitle
\begin{abstract}
Cholesky-QR is among the fastest algorithms for computing the thin QR
factorization of tall-and-skinny matrices on GPUs, relying entirely on
BLAS-3 operations. However, it is numerically unstable: forming the
Gram matrix squares the condition number, causing breakdown when
$\kappa_2(\boldsymbol{A}) \gtrsim 10^8$. We present \texttt{MRCQR} (Mixed-Precision
Randomized Cholesky-QR), a stable GPU algorithm that addresses this
limitation. \texttt{MRCQR} uses a subsampled randomized trigonometric
transform to construct a preconditioner $\boldsymbol{R}_s$ that reduces
$\kappa_2(\boldsymbol{A}\boldsymbol{R}_s^{-1})$ to near unity with high probability, then
applies Cholesky-QR in double precision to the preconditioned matrix.
The key insight---supported by perturbation analysis---is that the
preconditioner requires far less accuracy than the final result:
single (FP32) precision suffices when $\kappa_2(\boldsymbol{A}) \lesssim 10^8$,
and half (FP16) when $\kappa_2(\boldsymbol{A}) \lesssim 10^4$.
\texttt{MRCQR} produces an explicit orthogonal factor $\widehat{\boldsymbol{Q}}$
satisfying $\|\boldsymbol{I} - \widehat{\boldsymbol{Q}}^\top\widehat{\boldsymbol{Q}}\|_2 = \cal O(\mathbf{u})$
($\mathbf{u} \approx 10^{-16}$, double-precision unit roundoff)
for condition numbers up to $10^{16}$, far beyond the $10^8$ limit of
\texttt{CholQR2}. Experiments on an NVIDIA H100 GPU show that
\texttt{MRCQR} (FP16) outperforms \texttt{rand-cholQR}
by $1.4$--$1.8\times$ across all tested column counts and is
$1.8$--$13.5\times$ faster than cuSOLVER \texttt{geqrf}, while the
FP16 sketch (used when $\kappa_2(\boldsymbol{A}) \lesssim 10^4$) is $2\times$
cheaper than FP64 at no accuracy cost.

\end{abstract}


\section{Introduction}


Given a matrix $\ma\in\rmn$ with $m\gg n$ and $\rank(\ma)=n$, we
are interested in computing the thin QR factorization
\[\ma=\mq \mr, \quad \mq^\top \mq = \mi_{n},\]
where $\mq\in\rmn$ has orthonormal columns and $\mr\in\rnn$ is upper
triangular. In this paper we target double-precision output, though
the same principle generalizes to other working precisions. Such computation is required for dimension reduction
techniques~\cite{constantine2011tall}, the solution of linear least
squares problems~\cite{BjorckBook1996}, and inner loops of Krylov
space methods~\cite{Saad2003}, among others.

In several of these settings---particularly Krylov subspace methods
and randomized matrix approximation---$\mq$ must be formed
\emph{explicitly} as a dense matrix, since it serves directly as an
orthonormal basis for subsequent computations.
The orthogonality of this explicit $\mq$ is critical: in Krylov
subspace methods, for instance, loss of orthogonality among basis
vectors causes spurious eigenvalues and premature stagnation of
convergence~\cite{Saad2003}.
Moreover, many of these applications require the factorization to be
performed \emph{repeatedly}---the block Arnoldi process orthogonalizes
a new block of vectors at each iteration~\cite{Saad2003}, and
randomized SVD via power iteration applies a QR factorization at
each of its power steps~\cite{HMT11}---so per-call cost
accumulates directly.
Together, these demands call for a QR routine that produces a
numerically orthonormal explicit $\mq$ and is fast enough for
repeated use on a GPU.



\subsection{Background on Householder- and Cholesky-QR Methods}

The performance of QR factorization on cache-based architectures has long been driven by maximizing BLAS-3 usage: the blocked \texttt{geqrf} algorithm casts the trailing matrix updates into Level-3 BLAS, while the panel factorization remains a BLAS-2 bottleneck~\cite[Section 5.2]{GovL13}. On modern GPUs, where the throughput gap between BLAS-3 and BLAS-2 operations is far more pronounced, this panel bottleneck becomes the dominant performance limiter, motivating algorithms that eliminate it entirely. In practice, GPU users typically rely on cuSOLVER's \texttt{geqrf}/\texttt{orgqr}~\cite{cuSOLVER}, which carry this same bottleneck onto the GPU. While numerically stable, this approach underutilizes GPU parallelism for tall-and-skinny matrices where $n \ll m$.

The Tall-Skinny QR (TSQR) algorithm~\cite{DGHL08} improves on this by organizing the computation as a tree of local Householder QR factorizations, achieving communication-optimality and numerical stability. However, TSQR has further practical limitations on GPU hardware: it is not exposed in vendor libraries such as cuSOLVER, and forming an explicit $\mq$ requires an additional backward pass through the tree, increasing data movement. In practice, TSQR is frequently slower than simpler alternatives on GPUs~\cite{FukayaTSQR}, and even authors of optimized GPU implementations do not recommend it for matrices with more than approximately 32 columns on Nvidia H100 GPU due to its limited amount of shared memory~\cite{thies2026implementationqrfactorizationtall}.

A communication-avoiding, more GPU-friendly approach is Cholesky-QR, which replaces
Householder reflections with BLAS-3 matrix operations that map
efficiently onto GPU architectures. We present the basic Cholesky-QR
algorithm as \cref{alg:cholqr}. It reduces the problem to just three vendor-library calls: a
single matrix multiplication forming the (small) Gram matrix
$\ma^\top\ma$, followed by a small Cholesky factorization and a
triangular solve.
For tall-and-skinny matrices with $n\ll m$, the dominant cost is the
Gram matrix formation (step~1) and the triangular solve (step~3).
The former is a compute-bound GEMM that maps well onto GPU tensor
cores~\cite{StathWu2002}. The triangular solve is formally BLAS-3,
but for small $n$ its sequential dependency structure makes it
memory-bandwidth bound in practice, limiting GPU utilization.
Despite this bottleneck, Cholesky-QR remains highly competitive when
communication dominates arithmetic~\cite{Yama2015}, a common regime
in large-scale GPU computing.

\begin{algorithm}
\caption{ Cholesky-QR (\texttt{CholQR})}\label{alg:cholqr}
\begin{algorithmic}[1]
\Require $\ma\in\rmn$ with $ m\gg n$ and $\rank(\ma)=n$
\Ensure  $\mq \in\rmn$ and upper triangular $\mr \in\rnn$
\State $\mg \gets \ma^\top\ma$ \Comment{(Small) Gram matrix}
\State Compute $\mr$ s.t.\ $\mg=\mr^\top\mr$ \Comment{Cholesky factorization}
\State $\mq \gets \ma\mr^{-1}$ \Comment{Triangular solve}
\end{algorithmic}
\end{algorithm}

 In exact arithmetic, it is straightforward to verify that the output of \cref{alg:cholqr} $\mq = \ma\mr^{-1}$ satisfies $\mq^\top \mq = \mr^{-\top}\ma^\top\ma\mr^{-1} = \mr^{-\top}\mr^\top\mr\mr^{-1} = \mi_n$, i.e., the columns of $\mq$ are exactly orthonormal. However, in floating-point arithmetic, forming $\ma^\top\ma$ squares the condition number of $\ma$, and thus the computed $\widehat{\mq}$ can deviate significantly from orthonormality by $\mathcal{O}(\kappa_2(\ma)^2\,\vu)$~\cite{YNYF2015}, where $\kappa_2(\ma)$ is the ratio between the largest and the smallest singular values of $\ma$ and $\vu = 2^{-53} \approx 1.1 \cdot 10^{-16}$ is the double-precision unit roundoff~\cite{IEEE}. More critically, if $\kappa_2(\ma) \gtrsim \vu^{-1/2} \approx 10^8$, this squaring effect causes the computed Gram matrix to fail to be positive definite, so the Cholesky factorization breaks down entirely~\cite{Higham2002}. 
 
 The standard remedy, \texttt{CholQR2}~\cite{YNYF2015}, applies
\cref{alg:cholqr} twice to correct the orthogonality error from
$\mathcal{O}(\kappa_2(\ma)^2\vu)$ to $\mathcal{O}(\vu)$, but it
shares the same stability limit: both methods fail when
$\kappa_2(\ma) \gtrsim 10^8$, since the first pass requires the Gram
matrix to be positive definite. Shifted \texttt{CholQR3}~\cite{FKNYY2020} adds a stabilizing shift
to the Gram matrix, extending stability to
$\kappa_2(\ma) \lesssim 10^{16}$, but at approximately 50\% higher
cost than \texttt{CholQR2}. Randomized preconditioning
approaches~\cite{FanGuo2021,Balabanov2022,HSBY2023,MBMDML24} achieve
the same stability range by replacing the extra passes in \texttt{CholQR2} with an
inexpensive sketch, at roughly the same cost as \texttt{CholQR2}.
In particular, \texttt{rand-cholQR}~\cite{HSBY2023} uses a two-level
sparse-plus-dense multisketch and is the strongest prior randomized
baseline.  The authors of ~\cite{chen2025gpuparallelizablerandomizedsketchandpreconditionlinear} give a thorough benchmarking of a GPU-based sketch-and-precondition least squares solver for tall and skinny matrices based on a sparse sign embedding.  Our proposed method builds on these randomized approaches by
additionally exploiting mixed precision in the sketch phase to reduce
cost without sacrificing stability or accuracy.
The idea of computing a randomized preconditioner in reduced precision
has been explored for iterative solvers: Georgiou et
al.~\cite{georgiou2023lsqr} show GPU speedups for \texttt{LSQR}, and
Carson and Dau\v{z}ickaite~\cite{carson2024mixed} analyze
mixed-precision sketching for \texttt{GMRES}-based iterative
refinement.

\texttt{MRCQR} targets a different setting---direct QR
factorization with explicit $\mq$ formation---where the accuracy
metric is orthogonality rather than iterative convergence, and the
dominant cost after preconditioning is triangular solves and Gram
matrix formation on the full tall-and-skinny matrix.

\subsection{Our Approach and Contributions}
\label{sec:contributions}

Our proposed method, \textit{Mixed Precision Randomized Preconditioned
Cholesky-QR} (\texttt{MRCQR}), constructs an upper triangular preconditioner $\mr_s$
such that $\ma\mr_s^{-1}$ is well-conditioned, then applies a single
pass of \cref{alg:cholqr} to produce an explicit $\widehat{\mq}$ satisfying
\begin{equation}
\label{eq:orth}
\|\mi - \widehat{\mq}^\top\widehat{\mq}\|_2 = \mathcal{O}(\vu)
\end{equation}
 Specifically, $\mr_s$ is the $R$-factor
from the QR decomposition of $\mathbf{\Omega}\ma \in \real^{c \times
n}$ with $n \leq c \ll m$, where $\mathbf{\Omega}$ is a randomized sketching
matrix~\cite{HMT11}. This preconditioning idea was introduced for
least-squares by Rokhlin and Tygert~\cite{RT08,Blendenpik} and has been applied
to \texttt{CholQR} in prior
work~\cite{FanGuo2021,Balabanov2022,HSBY2023,MBMDML24}; the novelty
of \texttt{MRCQR} lies in exploiting mixed precision for the sketch.

\textbf{Contribution 1: mixed-precision sketch.} The main observation of this work is that $\mr_s$ need not be computed in full precision when $\kappa_2(\ma)$ is modest, without sacrificing any final accuracy. Targeting double-precision orthogonality $10^{-16}$ in \cref{eq:orth}, our perturbation analysis in \cref{rscond} establishes that:
\begin{itemize}
\item single precision suffices to construct $\mr_s$ when $\kappa_2(\ma) \lesssim 10^{8}$, and
\item half precision suffices when $\kappa_2(\ma) \lesssim 10^{4}$.
\end{itemize}
Given a coarse estimate of $\kappa_2(\ma)$, one can therefore select the lowest viable precision for the sketch, reducing both computation and memory traffic with no loss in final orthogonality. This observation is general and can be applied to improve the performance of other preconditioned \texttt{CholQR} methods.

\textbf{Contribution 2: extended numerical stability.}
\texttt{MRCQR} achieves near-machine-precision orthogonality
$\|\mi - \widehat{\mq}^\top\widehat{\mq}\|_2 = \mathcal{O}(\vu)$ for
$\kappa_2(\ma)$ up to $10^{16}$---matching the stability range of
\texttt{CholQR3}~\cite{FKNYY2020} but without its 50\% cost overhead,
and far beyond the $10^8$ limit shared by \texttt{CholQR} and
\texttt{CholQR2}~\cite{YNYF2015}. \texttt{MRCQR} remains
backward-stable across all tested condition numbers. Experiments
sweeping $\kappa_2(\ma) \in \{10^2, \ldots, 10^{16}\}$ confirm the
theoretical precision thresholds: the single-precision sketch maintains
$\mathcal{O}(\vu)$ orthogonality for $\kappa_2(\ma) \lesssim 10^8$
and the half-precision sketch for $\kappa_2(\ma) \lesssim 10^4$.

\textbf{Contribution 3: efficient GPU implementation.}
\texttt{MRCQR} is built entirely on vendor-optimized libraries
(\texttt{cuFFT}, \texttt{cuBLAS}, \texttt{cuSOLVER}~\cite{cuSOLVER}),
requiring no custom GPU kernels and making it immediately deployable
on any CUDA device. Profiling on the NVIDIA H100 reveals that the two
triangular solves (TRSM) dominate runtime (56--74\% for
$n = 16$--$128$) while the sketch phase accounts for only
10--20\%---the only component that changes with precision. Switching
from double to half-precision sketch reduces sketch cost by
${\approx}2\times$ at no accuracy cost, and \texttt{MRCQR} (FP16)
outperforms \texttt{rand-cholQR}~\cite{HSBY2023} by $1.4$--$1.8\times$
and cuSOLVER \texttt{geqrf} by $1.8$--$13.5\times$.

\section{Our Mixed-Precision Algorithm: \texttt{MRCQR}}\label{sec:MRCQR}

\texttt{MRCQR} constructs an upper triangular preconditioner $\mr_s$
--- the $R$-factor from the thin QR decomposition of $\mathbf{\Omega}\ma
\in \real^{c \times n}$ with $n \leq c \ll m$ --- such that
$\kappa_2(\ma\mr_s^{-1}) = \mathcal{O}(1)$ with high
probability~\cite{Blendenpik,IpsW12}. Applying
\texttt{CholQR} to the preconditioned matrix then yields
$\|\mi - \widehat{\mq}^\top\widehat{\mq}\|_2 = \mathcal{O}(\vu)$~\cite{garrisonipsen}.

\begin{algorithm}
\caption{\texttt{MRCQR}: Mixed-precision, Randomized-preconditioned Cholesky-QR}
\label{alg:mrcqr}
\begin{algorithmic}[1]
\Require $\ma\in\rmn$ with $ m\gg n$ and $\rank(\ma)=n$;
         sketch size $c$ (default $c=3n$~\cite{garrisonipsen});
         sketch precision $p \in \{\texttt{half},\, \texttt{single},\, \texttt{double}\}$
\Ensure  $\mq \in\rmn$ and upper triangular $\mr \in\rnn$

\Statex \medskip\textbf{Stage 1: Compute preconditioner in precision $p$}

\State Draw random Rademacher diagonal $\md \in \{+1,-1\}^m$
       and row indices $\mathcal{S} \subset \{1,\ldots,\lfloor m/2\rfloor+1\}$,
       $|\mathcal{S}| = c$
       \LineComment{Only $\lfloor m/2\rfloor+1$ unique bins due to Hermitian symmetry of real FFT}

\Statex
\State \textbf{[DiagMult]}~$\tilde{\ma} \gets \mathrm{fl}_p(\md \cdot \ma)$
\LineComment{Sign flip + cast $\ma$ to precision $p$}

\Statex
\State \textbf{[FFT]}~$\hat{\ma} \gets \mathrm{FFT}(\tilde{\ma})$
\LineComment{Column-wise real-to-complex FFT in precision $p$}

\Statex
\State \textbf{[Select]}~$\ma_s \gets \sqrt{c/m}\cdot\mathrm{Re}(\hat{\ma})[\mathcal{S},\,:]$
\LineComment{Sample $c$ rows; take real parts}

\Statex
\State \textbf{[SketchQR]}~$[\mq_s, \mr_s] \gets \textit{thinQR}(\ma_s)$
\LineComment{QR in precision $p$}

\Statex
\State Promote $\mr_s$ to double precision

\Statex \medskip\textbf{Stage 2: \texttt{CholQR} in double precision}

\State $\ma_1 \gets \ma\mr_s^{-1}$
\Comment{$\ma_1$ is well-conditioned}

\State $[\mq,\mr_2] \gets \texttt{CholQR}(\ma_1)$ \Comment{\cref{alg:cholqr}}

\Statex \medskip\textbf{Stage 3: Recover $\mr$}

\State $\mr \gets \mr_2 \mr_s$
\Comment{Upper-triangular factor of $\ma$}

\end{algorithmic}
\end{algorithm}

\subsection{Algorithm Description}\label{sec:precond}

\texttt{MRCQR} uses the SRTT~\cite{Blendenpik} as the sketching
operator $\mathbf{\Omega} = \ms\mf\md$, where $\md$ is a diagonal
Rademacher matrix, $\mf$ applies column-wise FFTs, and $\ms$
subsamples $c$ rows. We now explain each implementation choice.

Applying $\md$ requires no floating-point multiplications: since $d_i
\in \{+1,-1\}$, the sign flip reduces to XOR-ing the IEEE-754 sign
bit (the MSB: bit~63 for FP64, bit~31 for FP32, bit~15 for FP16, using
IEEE-754 numbering from the most significant end) and is fused with the
precision cast in a single pass over $\ma$. The diagonal multiply achieves modest
speedups: since $\ma$ must always be read in its original double
precision (8 bytes per element) regardless of the output precision,
the read traffic is fixed and the bandwidth savings are
fundamentally limited.

We use the FFT rather than the Walsh--Hadamard Transform (WHT), despite
WHT requiring no floating-point multiplications. Our benchmarks show
WHT is $3\times$ slower than \texttt{cuFFT} at $m = 2^{20}$, an
advantage we attribute to \texttt{cuFFT}'s engineering maturity rather
than any algorithmic difference.
Since $\ma$ is real-valued, its $m$-point DFT satisfies the Hermitian
symmetry $\hat{A}[m-k] = \overline{\hat{A}[k]}$ for $k = 1, \ldots,
\lfloor m/2 \rfloor$, so only $\lfloor m/2\rfloor + 1$ unique
frequency bins exist per column. We select $c$ bins and collect their
real parts, forming a $c\times n$ real sketch $\ms$ with
$S_{ij} = \operatorname{Re}(\hat{A}[k_i,j])$.
An alternative is to retain both real and imaginary parts, 
yielding a $2c\times n$ sketch at no extra FFT
cost but doubling the sketch QR cost. Our experiments show that the
Re-only variant achieves double-precision orthogonality
$\|\widehat{\mq}^\top\widehat{\mq} - \mi\|_2 \approx \vu$ across
all tested condition numbers, while halving the sketch QR cost.
We use $c = 3n$ samples, giving a $3\times$ overdetermined sketch;
this choice is based on extensive empirical experiments
in~\cite{garrisonipsen}, which found it to 
reliably produces a well-conditioned preconditioner $\mr_s$
in practice.

For the small $n$ values considered here, sketch QR accounts for a
small fraction of total sketch time. Nevertheless, two limitations
are worth noting. First, \texttt{Sgeqrf} achieves only
${\sim}1.09\times$ speedup over \texttt{Dgeqrf}: cuSOLVER's blocked
Householder QR interleaves a sequential BLAS-2 panel factorization
that costs the same regardless of precision, creating an Amdahl
bottleneck. Second, cuSOLVER provides no half-precision
\texttt{geqrf}; the half-precision sketch variant therefore promotes
the sketch matrix to single precision and executes the same
\texttt{Sgeqrf} call, giving identical SketchQR times for both
precision levels.

In Stage~2, $\mr_s$ is promoted to double precision before the
triangular solve $\ma_1 \gets \ma\mr_s^{-1}$: performing this solve
in reduced precision would incur a backward error proportional to
$\vu_p$ on the full $m\times n$ matrix, degrading the orthogonality
achieved by the subsequent double-precision \texttt{CholQR}.
Within \texttt{CholQR}
we use \texttt{dgemm} rather than \texttt{dsyrk} for the Gram matrix:
profiling shows \texttt{dsyrk} significantly underperforms
\texttt{dgemm} for tall-and-skinny shapes on the Nvidia H100 GPU. Finally, the
triangular solve uses \texttt{dtrsm} directly rather than explicitly
forming $\mr_s^{-1}$ and applying \texttt{dgemm}: explicitly forming
$\mr_s^{-1}$ squares the condition number in the error bound,
potentially losing all significant digits when $\kappa_2(\mr_s)$ is
non-trivial.

\subsection{Mixed-Precision Justification}\label{sec:math}

The SRTT sketching operator ensures that $\kappa_2(\ma\mr_s^{-1}) =
\mathcal{O}(1)$ with high probability~\cite{Blendenpik,IpsW12}.
Here $\ma_1 \equiv \ma\mr_s^{-1}$ is the ideal preconditioned matrix
in exact arithmetic. In finite arithmetic, computing $\mr_s$ in
reduced precision produces a perturbed factor $\mr_s + \me$, so the
actual computed object is $\widehat{\ma_1} \equiv \ma(\mr_s+\me)^{-1}$.
\Cref{rscond} shows that $\kappa_2(\widehat{\ma_1})$ remains
$\mathcal{O}(1)$ provided the relative error $\varepsilon \equiv
\|\me\|_2/\|\mr_s\|_2$ from computing $\mr_s$ in reduced precision
satisfies $\varepsilon < \kappa_2(\mr_s)^{-1}$.

A standard floating-point analysis gives
$\varepsilon = \mathcal{O}(\vu_p)$ for computing $\mr_s$ in
precision $\vu_p$. Since $\kappa_2(\mr_s) \approx \kappa_2(\ma)$,
the condition $\varepsilon < \kappa_2(\mr_s)^{-1}$ becomes
$\vu_p \cdot \kappa_2(\ma) \lesssim 1$, directly yielding the precision
thresholds in \cref{alg:mrcqr}: single precision ($\vu_s \approx
10^{-8}$) is safe when $\kappa_2(\ma) \lesssim 10^8$, and half
precision ($\vu_h \approx 10^{-4}$) when $\kappa_2(\ma) \lesssim
10^4$.
Once $\widehat{\ma_1}$ is well-conditioned, the analysis
of~\cite[Corollary 3.2]{garrisonipsen} guarantees that applying
\texttt{CholQR} to $\widehat{\ma_1}$ achieves $\|\mi -
\widehat{\mq}^\top\widehat{\mq}\|_2 = \mathcal{O}(\vu)$, 
where $\widehat{\mq}$ is the output of \cref{alg:mrcqr}, completing
the theoretical justification for \texttt{MRCQR}.

We now state this precisely.

\begin{theorem}\label{rscond} 
Assume $\ma$ has full column rank and $\mr_s$ is non-singular.
Let $\ma_1\equiv\ma\mr_s^{-1}$,
$\widehat{\ma_1}\equiv \ma(\mr_s+\me)\inv$, and
$\varepsilon\equiv\frac{\|\me\|_2}{\|\mr_s\|_2}$.
If $\|\me\|_2\|\mr_s^{-1}\|_2<1$, then
 \[\kappa_2(\widehat{\ma_1})\leq\kappa_2(\ma_1)\>\frac{1+\varepsilon\ \kappa_2(\mr_s)}{1-\varepsilon\ \kappa_2(\mr_s)}. \]
\end{theorem}

The assumption $\|\me\|_2\|\mr_s\inv\|_2 < 1$ holds whenever
$\kappa_2(\ma) \lesssim \vu^{-1} \approx 10^{16}$, as confirmed
by our experiments.

\begin{proof}
Since $\|\me\|_2\|\mr_s^{-1}\|_2<1$, the Banach lemma
\cite[Lemma 2.3.3]{GovL13} implies that $\mr_s+\me$ is nonsingular,
so that $\widehat{\ma_1}$ is well defined.
To relate $\widehat{\ma_1}$ to $\ma_1$, factor out $\mr_s^{-1}$,
\begin{align*}
\widehat{\ma_1}=\ma\mr_s^{-1}(\mi+\me\mr_s^{-1})^{-1}
=\ma_1(\mi+\me\mr_s^{-1})^{-1}.
\end{align*}
The Banach lemma also implies
\begin{align*}
\|\widehat{\ma_1}\|_2\leq
\frac{\|\ma_1\|_2}{1-\|\me\|_2\|\mr_s^{-1}\|_2}
=\frac{\|\ma_1\|_2}{1-\varepsilon\ \kappa_2(\mr_s)}.
\end{align*}
From $\widehat{\ma_1}$ having full column rank follows that its
left inverse equals
\begin{align*}
\widehat{\ma_1}^{\dagger}=(\mi+\me\mr_s^{-1})\ma_1^{\dagger}.
\end{align*}
Analogous to the above, we bound
\begin{align*}
\|\widehat{\ma_1}^{\dagger}\|_2&\leq
\|\ma_1^{\dagger}\|_2\ (1+\|\me\|_2\|\mr_s^{-1}\|_2)\\
&=\|\ma_1^{\dagger}\|_2\ (1+\varepsilon\ \kappa_2(\mr_s)).
\end{align*}
Combining gives
\begin{align*}
\kappa_2(\widehat{\ma_1})&\leq
\kappa_2(\ma_1)\ \frac{1+\varepsilon\ \kappa_2(\mr_s)}{1-\varepsilon\ \kappa_2(\mr_s)}.
\end{align*}
\end{proof}

\subsection{Adaptive Precision Selection}\label{sec:adapt}

The sketch precision $p$ is the only user-supplied parameter in
\texttt{MRCQR}, and only a \textit{coarse} classification of
$\kappa_2(\ma)$ into one of three regimes is required to set it:
$\kappa_2(\ma) < 10^4$ (use half), $10^4 \leq \kappa_2(\ma) < 10^8$
(use single), or $\kappa_2(\ma) \geq 10^8$ (use double). The precise
value of $\kappa_2(\ma)$ is never needed. In many practical settings
this regime is known from prior knowledge of the problem---for
instance, from the mesh size in a discretized PDE or the spectrum of
a covariance matrix---and no estimation is required (see
\cref{fig:variouscond} for the effect of an incorrect choice).

When the regime is not known a priori, several options are available.
The cheapest is a \emph{CholQR probe}: apply \cref{alg:cholqr} to
$\ma$ and invert the orthogonality deviation $q \equiv \|\mi -
\widehat{\mq}^\top\widehat{\mq}\|_2$ using the bound

\begin{equation}
\kappa_2(\ma) \geq \sqrt{\frac{q}{\vu(1+q)}},
\end{equation}
derived from~\cite[Theorem 2.2]{garrisonipsen}. This probe costs one \texttt{CholQR} call and is essentially free
when \texttt{CholQR} was already the intended first attempt. Alternatively, the block 1-norm estimator
of~\cite{doi:10.1137/S0895479899356080} applied to $\ma^\top\ma$
classifies the regime at $\mathcal{O}(n^2)$ cost, exploiting the
bound $\kappa_2(\ma)^2 \leq n\|\ma^\top\ma\|_1
\|(\ma^\top\ma)^{-1}\|_1$~\cite[Exercise 2.6]{IIbook}. Machine-learning-based estimators~\cite{carson2025condest} offer a
further avenue when a training distribution is available.
A full implementation of condition number estimation is beyond the
scope of this work; our experiments fix $\kappa_2(\ma)$ at
construction time. Developing an integrated, GPU-friendly estimator
is a natural direction for future work.

When QR is performed \emph{repeatedly} on a sequence of similar
matrices---such as the
block Arnoldi process~\cite{Saad2003}, randomized SVD via power
iteration~\cite{HMT11}, or time-stepping schemes---the conditioning
regime typically changes slowly or predictably. A single regime
classification then applies to many subsequent calls, amortizing its
cost to a negligible fraction of the total computation.

If the precision is underestimated, the deviation from orthonormality
will exceed $\mathcal{O}(\vu)$, which is immediately detectable by
checking $\|\mi - \widehat{\mq}^\top\widehat{\mq}\|_2$. The remedy is
to rerun \texttt{MRCQR} at the next higher precision level. This
self-correcting property means that an aggressive (low-precision)
initial guess carries no risk beyond one wasted factorization.

\section{Numerical Experiments}\label{sec:experiments}

\subsection{Experimental Setup}\label{sec:implement}

\texttt{MRCQR} is implemented in C\texttt{++}/CUDA using
vendor-optimized libraries: cuBLAS for GEMM and triangular solves,
cuSOLVER for QR and Cholesky factorizations, and cuFFT for
column-wise FFTs. Half-precision FFTs use the \texttt{cufftXt} API
(\texttt{CUDA\_R\_16F} input, \texttt{CUDA\_C\_16F} output), since
the standard \texttt{cufftExecR2C} does not expose a half-precision
interface. All results were obtained with CUDA~13.1 and GCC~11.5.0
on an NVIDIA H100 PCIe (80\,GB VRAM); reported times exclude
host--device transfers and are averaged over 5 trials.

We compare against three baselines: (1)~cuSOLVER \texttt{geqrf},
the standard industry-stable GPU QR baseline; (2)~\texttt{CholQR2}~\cite{YNYF2015},
implemented by calling our \texttt{CholQR} routine twice, included
for stability comparison; and
(3)~\texttt{rand-cholQR}~\cite{HSBY2023}, a state-of-the-art
randomized preconditioned \texttt{CholQR} using a two-level
sparse-plus-dense multisketch, provided by the authors and built
with Kokkos~5.1.1~\cite{Kokkos2022}.

Our implementation of MRCQR can be found at \textbf{\url{https://github.com/as1805/MRCQR}}.

\subsection{Numerical Stability}\label{sec:adaptive accuracy}

We assess the orthogonality and backward stability of \texttt{MRCQR}
against \texttt{geqrf} (the stable reference) and \texttt{CholQR2}
(known to break down for $\kappa_2(\ma) \gtrsim 10^8$~\cite{YNYF2015})
across a wide range of condition numbers.
\texttt{rand-cholQR} is omitted here as its stability is established
in~\cite{HSBY2023}.

We construct test matrices $\ma \in \rmn$ with prescribed condition
numbers via $\ma = \mU\boldsymbol{\Sigma}\mv^\top$, where $\mU$,
$\mv$ have orthonormal columns (obtained by QR of random matrices
with entries uniform on $[-1,1]$) and the singular values are
geometric:
\begin{equation}
    \sigma_i = \kappa^{\,1/2 - i/(n-1)}, \quad i = 0, \ldots, n-1,
\end{equation}
giving $\kappa_2(\ma) = \kappa$ exactly. We fix $m = 2^{17}$,
$n = 50$, and sweep $\kappa \in \{10^2, 10^3, \ldots, 10^{16}\}$.

\begin{figure}
    \centering
    \begin{subfigure}[b]{0.45\textwidth}
        \centering
        \includegraphics[width=.85\linewidth]{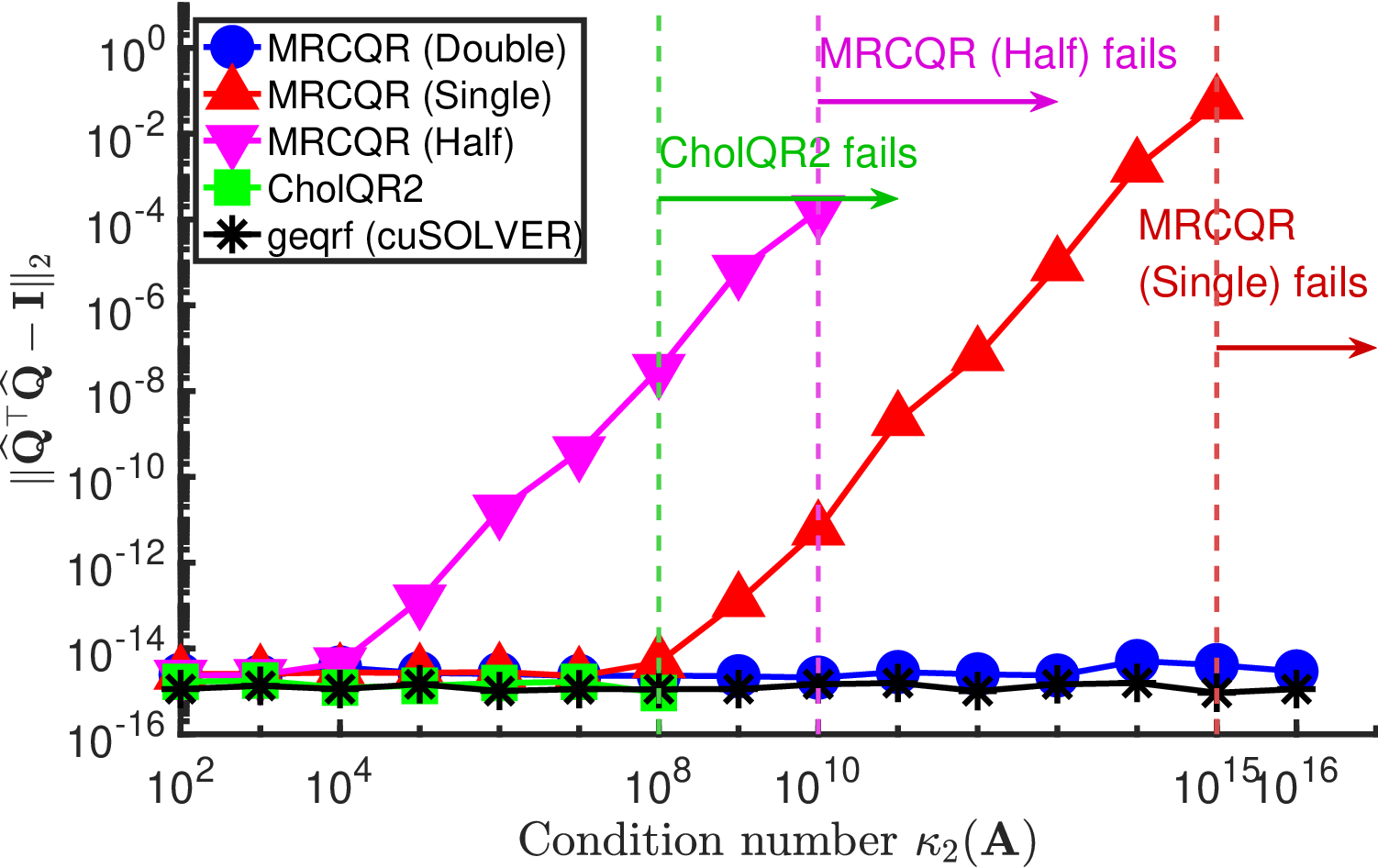}
        \caption{Deviation from orthonormality $\|\widehat{\mq}^\top\widehat{\mq} - \mi\|_2$
                 vs.\ $\kappa_2(\ma)$}
        \label{fig:accuracy_vs_cond}
    \end{subfigure}
    \begin{subfigure}[b]{0.45\textwidth}
        \centering
        \includegraphics[width=.85\linewidth]{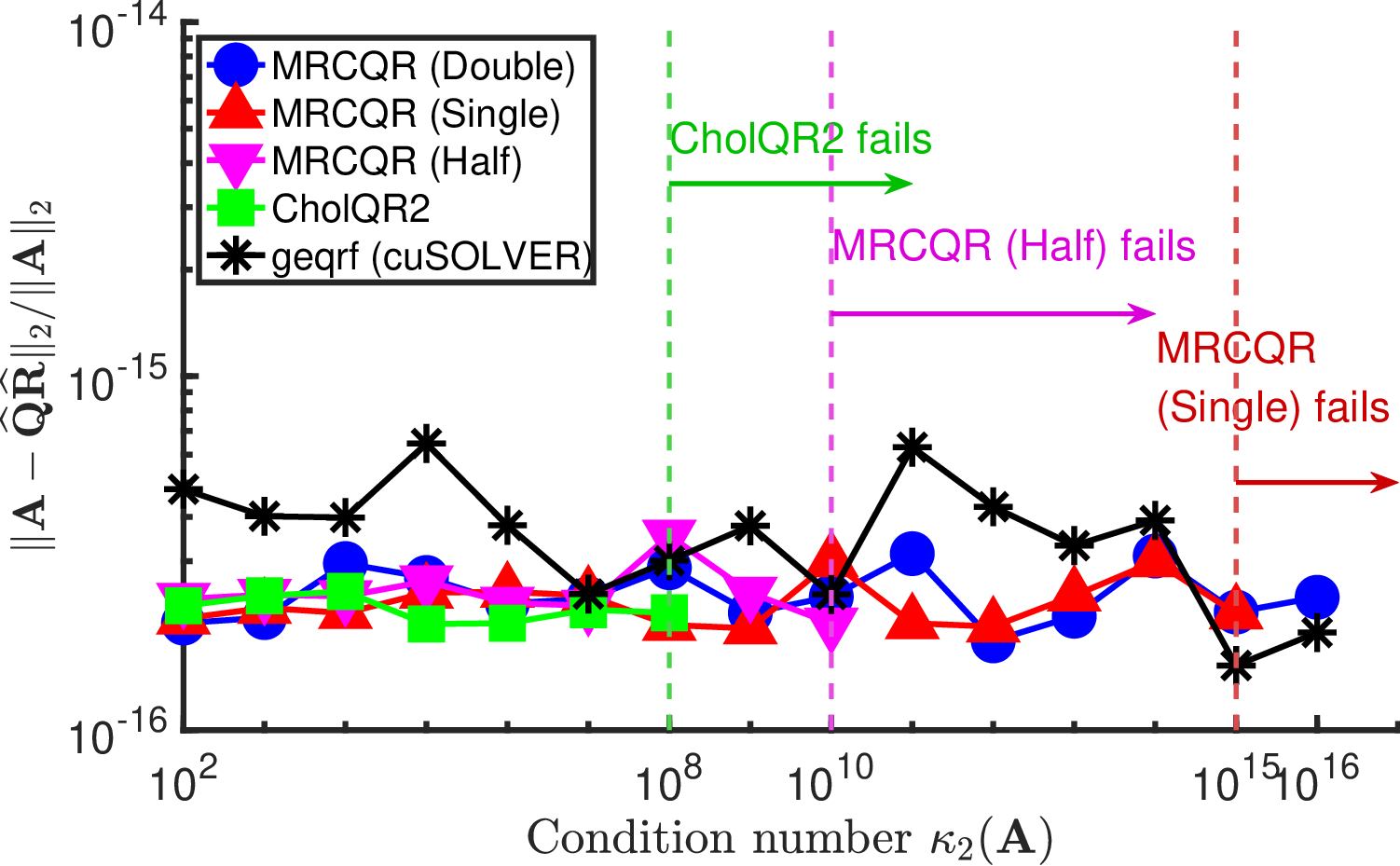}
        \caption{Relative residual $\|\ma - \widehat{\mq}\widehat{\mr}\|_2 / \|\ma\|_2$
                 vs.\ $\kappa_2(\ma)$}
        \label{fig:residual_vs_cond}
    \end{subfigure}
    \caption{\texttt{CholQR2} breaks down past $\kappa_2(\ma) = 10^8$
    (Gram matrix loses positive definiteness).
    \texttt{MRCQR} (half-precision sketch) breaks down past $\kappa_2(\ma) = 10^{10}$, and 
    \texttt{MRCQR} (single-precision sketch) breaks down past $\kappa_2(\ma) = 10^{15}$ (underflow of the
    smallest singular value of $\ma$).}
    \label{fig:variouscond}
\end{figure}

\Cref{fig:variouscond} shows deviation from orthonormality and
relative residual for all methods. \texttt{geqrf} achieves the gold standard of
$\|\mi - \widehat{\mq}^\top\widehat{\mq}\|_2 \approx 10^{-15}$
for all $\kappa$; \texttt{MRCQR} (double-precision sketch) matches
this at ${\approx} 5\times 10^{-15}$ across all tested condition numbers.
Confirming the thresholds of \cref{rscond}, \texttt{MRCQR}
(single-precision sketch) maintains $\mathcal{O}(\vu)$ orthogonality
for $\kappa_2(\ma) \lesssim 10^8$ and \texttt{MRCQR}
(half-precision sketch) for $\kappa_2(\ma) \lesssim 10^4$; beyond
these thresholds orthogonality degrades.
\texttt{CholQR2} breaks down past $\kappa_2(\ma) = 10^8$, consistent
with known analysis~\cite{YNYF2015}. Notably, \texttt{MRCQR}
(single-precision sketch) continues to produce a result beyond
$10^8$---albeit with degraded orthogonality---whereas \texttt{CholQR2}
fails entirely. This graceful degradation is useful in practice: a
user who underestimates $\kappa_2(\ma)$ receives a detectable warning
rather than a hard failure.
All methods that do not break down are backward stable
(\cref{fig:residual_vs_cond}).

\subsection{Runtime Performance}\label{sec:precisionchange}

We compare \texttt{MRCQR} (FP64, FP32, FP16 sketch) against
\texttt{rand-cholQR}~\cite{HSBY2023} and cuSOLVER \texttt{geqrf},
fixing $m = 2^{20}$ and varying $n$. All three methods are stable for $\kappa_2(\ma) \lesssim 10^{16}$,
making this a fair speed comparison. The sketch size for
\texttt{rand-cholQR} follows the authors'
recommendation~\cite[Section 6.2]{HSBY2023}:
$p_1 = \lceil 8.24(n^2+n)\rceil$ (CountSketch) and
$p_2 = \lceil 74.3\log(p_1)\rceil$ (Gaussian); note $p_1$ grows
quadratically in $n$, whereas \texttt{MRCQR} uses $c = 3n$.
Although the CountSketch is a sparse transform and is theoretically
efficient, realising this on GPU requires a highly optimized sparse
kernel; the Kokkos~\cite{Kokkos2022} implementation provided by the authors may not
reflect the full performance potential of \texttt{rand-cholQR}.

\begin{figure}
    \centering
    \begin{subfigure}[b]{0.45\textwidth}
        \centering
        \includegraphics[width=.85\linewidth]{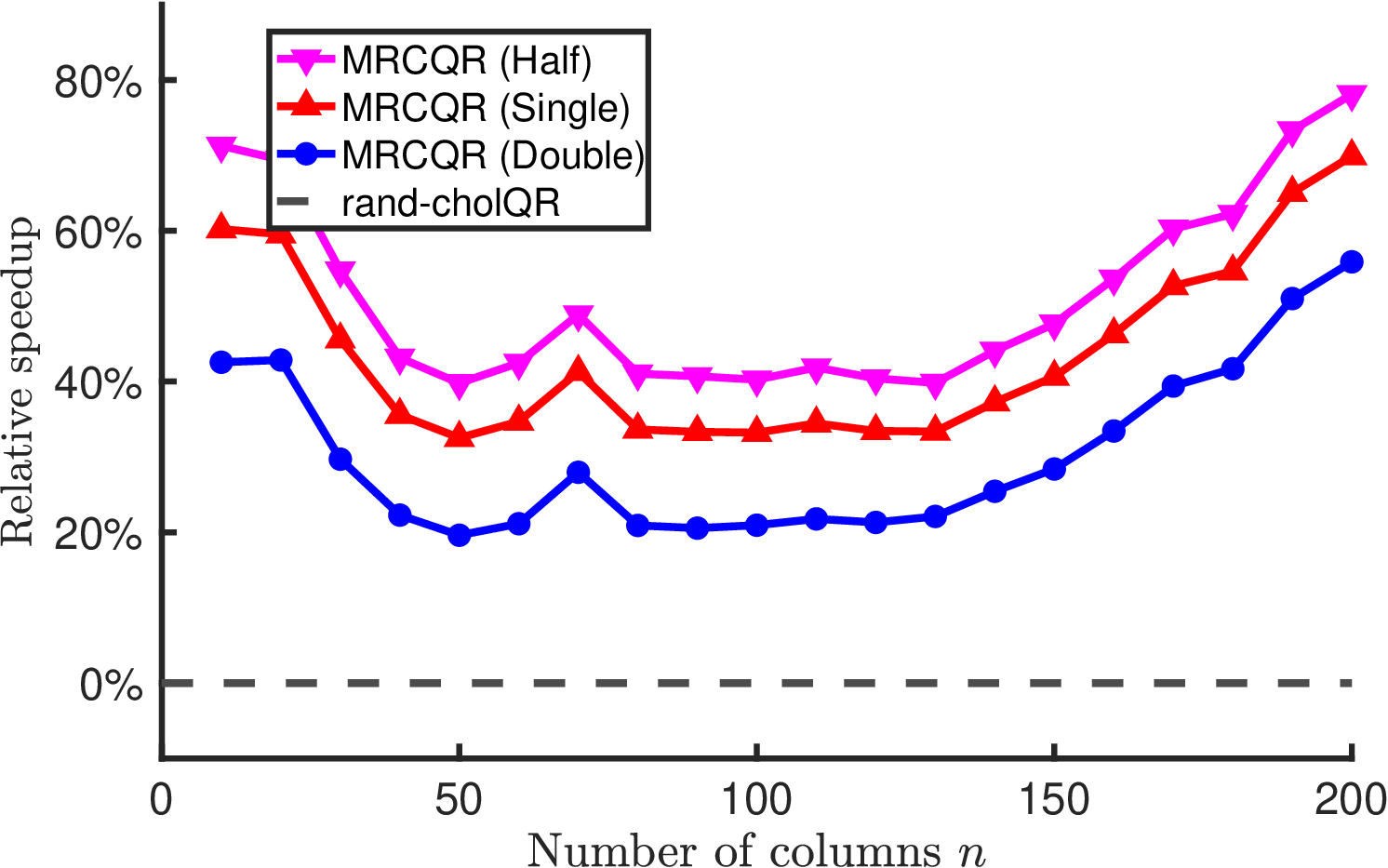}
        \caption{Speedup of \texttt{MRCQR} over \texttt{rand-cholQR} vs.\ $n$}
        \label{fig:perf_speedup}
    \end{subfigure}
    \begin{subfigure}[b]{0.45\textwidth}
        \centering
        \includegraphics[trim={0cm 0cm 0cm -1.5cm}, clip, width=.85\linewidth]{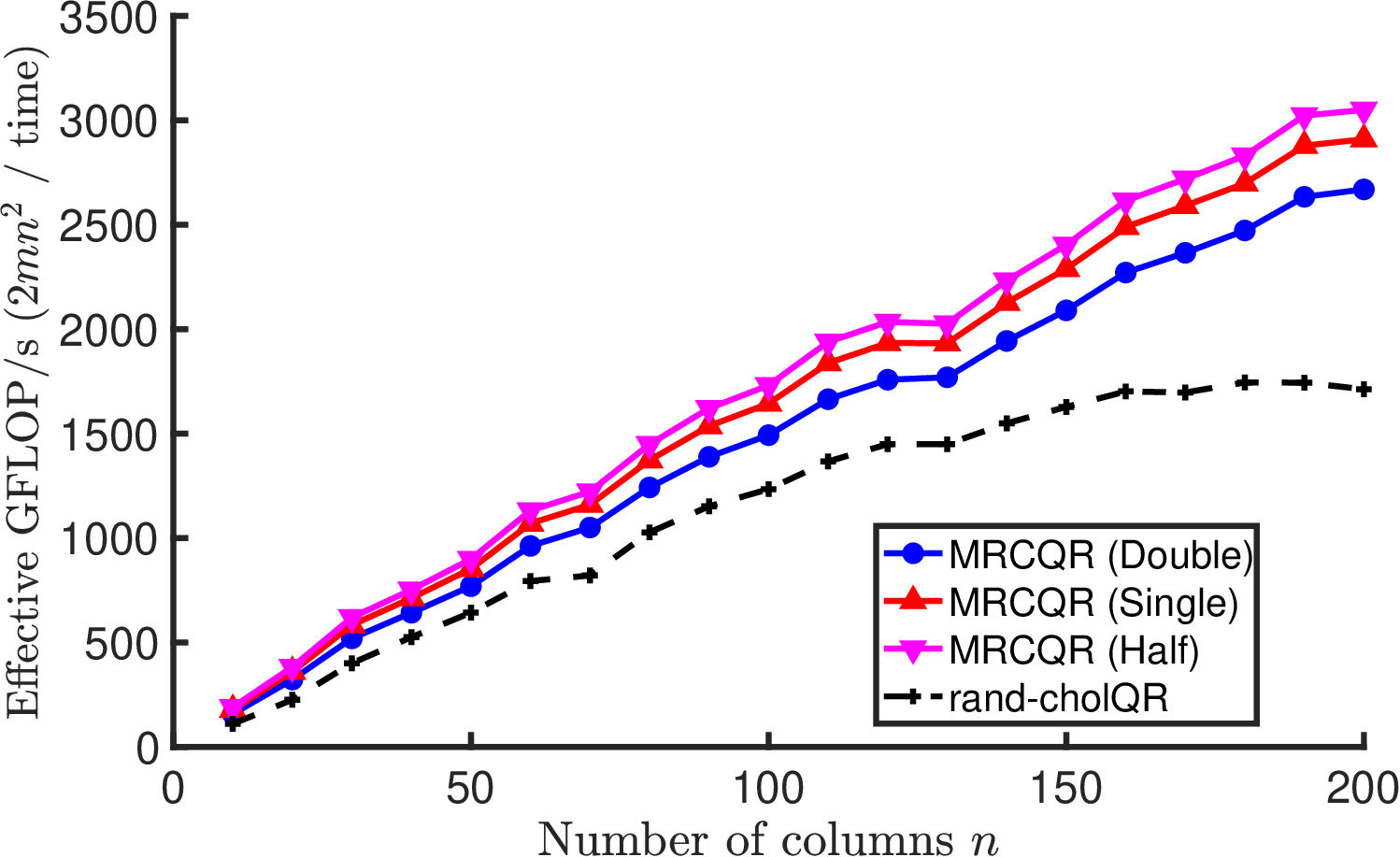}
        \caption{GFlop/s of \texttt{MRCQR} and \texttt{rand-cholQR} vs.\ $n$}
        \label{fig:perf_flops}
    \end{subfigure}
    \begin{subfigure}[b]{0.45\textwidth}
        \centering
        \includegraphics[trim={-3.5cm 0cm 2cm 0cm}, clip, width=\linewidth]{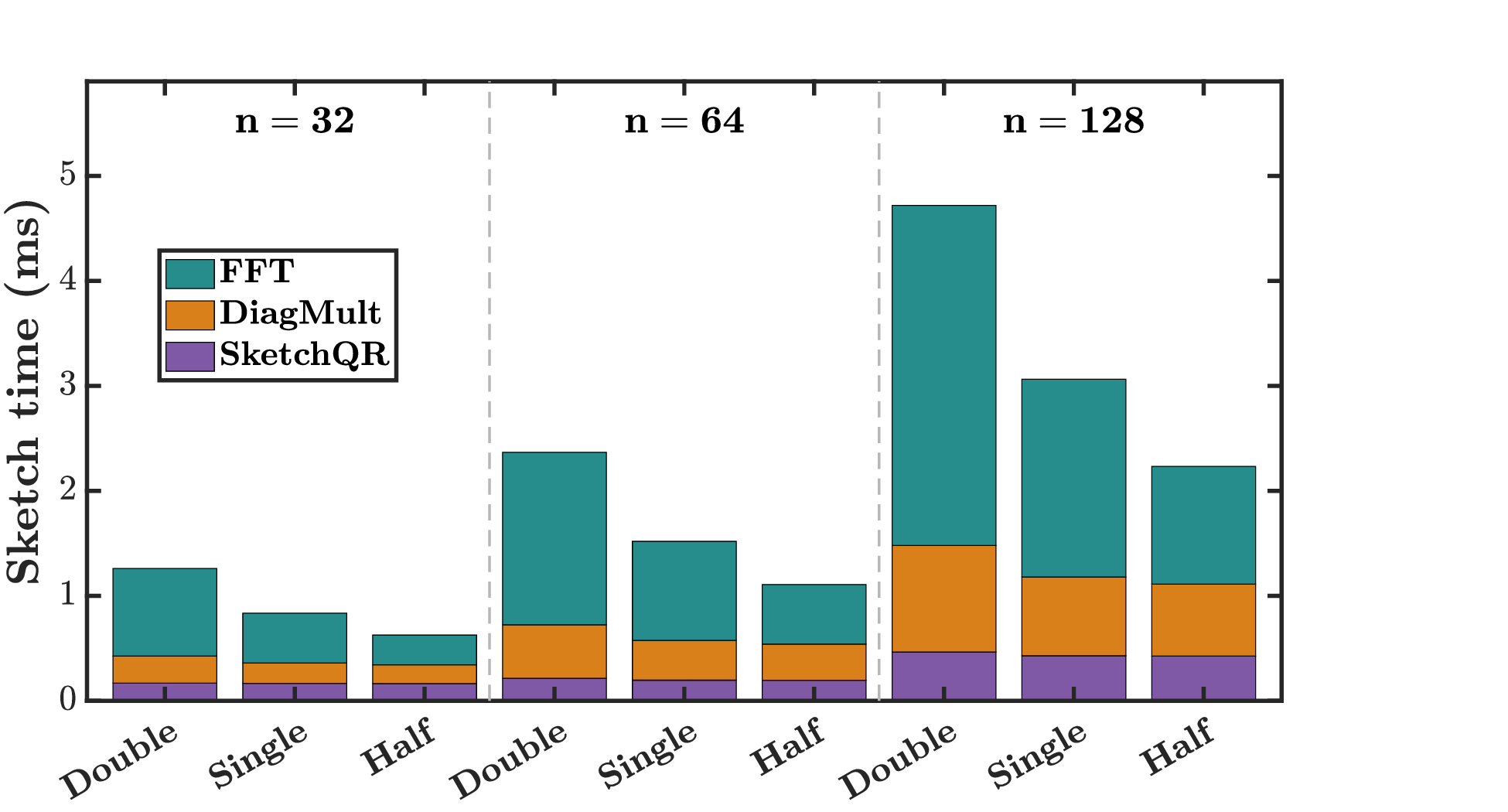}
        \caption{Sketch-phase timing breakdown across three precisions}
        \label{fig:sketch_precision_small}
    \end{subfigure}
    \caption{Performance of \texttt{MRCQR} vs.\ \texttt{rand-cholQR}~\cite{HSBY2023}
             on NVIDIA H100 GPU, $m = 2^{20}$, averaged over 5 trials.}
    \label{fig:perf}
\end{figure}

\Cref{fig:perf} illustrates three complementary aspects of
\texttt{MRCQR}'s performance.
\Cref{fig:perf_speedup} shows that all three precision variants
outperform \texttt{rand-cholQR}, with the gap increasing with $n$;
lower sketch precision yields greater speedup, confirming the
practical recommendation: use the lowest precision consistent with
$\kappa_2(\ma)$---half for $\kappa_2(\ma) \lesssim 10^4$,
single for $\kappa_2(\ma) \lesssim 10^8$, and double
otherwise.
\Cref{fig:perf_flops} shows the effective flop rates of
\texttt{MRCQR} and \texttt{rand-cholQR}: \texttt{MRCQR}
achieves substantially higher utilization, reflecting the benefit of
building on highly optimized vendor libraries (\texttt{cuFFT},
\texttt{cuBLAS}) rather than custom sparse kernels.
\Cref{fig:sketch_precision_small} shows the sketch-phase timing
breakdown versus $n$ for all three precisions; the ratios between
precision tiers remain approximately constant across $n$. The speedup from lower
precision is driven primarily by the FFT, which dominates the sketch
cost and achieves ${\sim}1.7\times$ per precision level.

\begin{table}[t]
\centering
\caption{Wall-clock time (ms) on H100, $m = 2^{20}$, 5 trials.
         Speedup = \texttt{rand-cholQR} or \texttt{geqrf} divided by
         \texttt{MRCQR}~(FP16).}
\label{tab:timing}
\setlength{\tabcolsep}{3pt}
\rowcolors{2}{gray!15}{white}
\begin{tabular}{c|ccc|c|c|cc}
\hline
\rowcolor{white}
 & \multicolumn{3}{c|}{\texttt{MRCQR}} & & & \multicolumn{2}{c}{Speedup vs.\ \texttt{MRCQR}~(FP16)} \\
\rowcolor{white}
$n$ & FP64 & FP32 & FP16 & \shortstack{\texttt{rand-}\\\texttt{cholQR}} & \texttt{geqrf} & \shortstack{vs \texttt{rand-}\\\texttt{cholQR}} & vs \texttt{geqrf} \\
\hline
 20  &  2.60 &  2.33 &  2.19 &  3.71 & 29.6 & 1.69$\times$ & 13.5$\times$ \\
 40  &  5.22 &  4.71 &  4.46 &  6.38 & 16.6 & 1.43$\times$ &  3.7$\times$ \\
 60  &  7.84 &  7.06 &  6.67 &  9.50 & 20.9 & 1.42$\times$ &  3.1$\times$ \\
 80  & 10.80 &  9.77 &  9.26 & 13.06 & 17.4 & 1.41$\times$ &  1.9$\times$ \\
100  & 14.05 & 12.75 & 12.11 & 16.99 & 23.0 & 1.40$\times$ &  1.9$\times$ \\
120  & 17.18 & 15.62 & 14.84 & 20.84 & 28.2 & 1.40$\times$ &  1.9$\times$ \\
140  & 21.15 & 19.32 & 18.41 & 26.53 & 36.8 & 1.44$\times$ &  2.0$\times$ \\
160  & 23.64 & 21.56 & 20.53 & 31.54 & 41.0 & 1.54$\times$ &  2.0$\times$ \\
180  & 27.48 & 25.18 & 24.00 & 38.94 & 53.5 & 1.62$\times$ &  2.2$\times$ \\
200  & 31.43 & 28.83 & 27.50 & 48.98 & 50.5 & 1.78$\times$ &  1.8$\times$ \\
\hline
\end{tabular}
\end{table}

\Cref{tab:timing} compares \texttt{MRCQR} (three precisions),
\texttt{rand-cholQR}, and cuSOLVER's \texttt{geqrf} on the H100.
Three observations stand out.
First, \texttt{MRCQR} (FP16) is $1.4$--$1.8\times$ faster than the
provided \texttt{rand-cholQR} implementation across all tested $n$,
with the gap widening as $n$ grows. This reflects not algorithmic
superiority but the ease of implementing \texttt{MRCQR} efficiently:
its SRTT-based sketch maps directly onto highly optimized vendor
libraries (\texttt{cuFFT}), requiring no custom sparse GPU kernel.
Second, the speedup of \texttt{MRCQR} (FP16) over \texttt{geqrf} is
$13.5\times$ at $n=20$: \texttt{geqrf} incurs a large fixed kernel
launch and synchronization overhead that dominates for small $n$,
while \texttt{MRCQR}'s SRTT-based sketch avoids this entirely.
At larger $n$ ($80$--$200$), where both methods are dominated by
$\mathcal{O}(mn^2)$ work, the speedup stabilizes at
$1.8$--$2.2\times$, reflecting \texttt{MRCQR}'s BLAS-3 efficiency
advantage.
Third, the FP16 sketch is ${\approx}2\times$ faster than FP64 at no accuracy
cost (for $\kappa_2(\ma) \lesssim 10^4$), confirming the
mixed-precision benefit quantified in \cref{tab:sketch_precision_n64}.

\begin{table}[t]
\centering
\caption{Sketch sub-operation timings (ms) at $n=64$, $m=2^{20}$.
	Ratios across precisions are approximately $n$-independent.}
\label{tab:sketch_precision_n64}
\setlength{\tabcolsep}{4pt}
\begin{tabular}{c|ccc|c}
\hline
Precision & DiagMult & FFT & SketchQR & Total \\
\hline
FP64 & 0.511 & 1.642 & 0.214 & 2.367 \\
FP32 & 0.379 & 0.943 & 0.197 & 1.518 \\
FP16 & 0.344 & 0.568 & 0.196 & 1.107 \\
\hline
FP64/FP32 & 1.35$\times$ & 1.74$\times$ & 1.09$\times$ & 1.56$\times$ \\
FP32/FP16 & 1.10$\times$ & 1.66$\times$ & 1.01$\times$ & 1.37$\times$ \\
\hline
\end{tabular}
\end{table}

\Cref{tab:sketch_precision_n64} breaks down sketch-phase timings at
$n=64$; the per-precision speedup ratios are approximately
$n$-independent. The FFT dominates sketch cost and achieves
${\sim}1.7\times$ speedup per precision level. DiagMult achieves
more modest gains (${\sim}1.35\times$ FP64$\to$FP32,
${\sim}1.10\times$ FP32$\to$FP16) because $\ma$ must always be read
in FP64, limiting bandwidth savings. SketchQR achieves only
${\sim}1.09\times$ speedup from FP64 to FP32 and no further gain
from FP32 to FP16, for the reasons discussed in \cref{sec:precond}.

\begin{table}[t]
\centering
\caption{Cost breakdown of \texttt{MRCQR} at $m=2^{20}$, 10 trials
         (ms). TRSM$\times$2 = TrsmPre\,+\,TrsmFin; GEMM = Gram
         matrix; Sketch = full SRTT phase. TRSM and GEMM costs are
         identical across precisions; only Sketch changes.}
\label{tab:breakdown}
\setlength{\tabcolsep}{4pt}
\begin{tabular}{c|cc|ccc}
\hline
$n$ & TRSM$\times$2 & GEMM & Sk.\ FP64 & Sk.\ FP32 & Sk.\ FP16 \\
\hline
 16 &  0.93 & 0.12 & 0.62 & 0.40 & 0.29 \\
 32 &  2.05 & 0.17 & 1.26 & 0.84 & 0.63 \\
 64 &  5.00 & 0.29 & 2.37 & 1.52 & 1.11 \\
128 & 11.65 & 0.68 & 4.72 & 3.06 & 2.23 \\
\hline
\end{tabular}
\end{table}

\Cref{tab:breakdown} reveals the cost structure of \texttt{MRCQR}.
Three points stand out.
First, the two triangular solves (TRSM) dominate the runtime: TRSM
accounts for 56\% of total time at $n=16$, growing to 74\% at
$n=128$. This is perhaps surprising given that TRSM is a BLAS-3
operation, but for tall-and-skinny matrices it is memory-bandwidth
bound---the sequential wavefront dependency limits GPU
parallelism---giving near-linear scaling in $n$.
Second, GEMM is negligible at these column counts (0.12--0.68\,ms),
confirming that the compute-intensive Gram matrix formation is not
the bottleneck in this regime.
Third, the Sketch cost is the only quantity that changes with
precision: FP16 reduces sketch cost by ${\approx}2\times$ over FP64
at each $n$, while TRSM and GEMM remain unchanged. This directly
confirms that mixed-precision sketching targets exactly the right
part of the algorithm.

\section{Conclusions}
\texttt{MRCQR} is a numerically stable GPU algorithm for the thin QR
factorization of tall-and-skinny matrices, achieving
$\|\mi - \widehat{\mq}^\top\widehat{\mq}\|_2 = \mathcal{O}(\vu)$
for $\kappa_2(\ma)$ up to $10^{16}$---far beyond the $10^8$ stability
limit of \texttt{CholQR2}. The central insight is that the SRTT-based
preconditioner only needs to reduce the effective condition number to
$\mathcal{O}(1)$, not achieve full double-precision accuracy: FP32
suffices when $\kappa_2(\ma) \lesssim 10^8$ and FP16 when
$\kappa_2(\ma) \lesssim 10^4$, exactly matching the predicted
thresholds. On an NVIDIA H100, \texttt{MRCQR} (FP16) is
$1.8$--$13.5\times$ faster than cuSOLVER \texttt{geqrf} and
$1.4$--$1.8\times$ faster than \texttt{rand-cholQR}, with no custom
GPU kernels. Profiling reveals that TRSM dominates at 56--74\% of
total runtime while the sketch accounts for only 10--20\%, so
tensor-core-based TRSM reformulation~\cite{Carrica2025} is the most
promising direction for further speedup. Precision selection is easily
determined from problem structure or a cheap \texttt{CholQR} probe,
making \texttt{MRCQR} practical in production HPC workflows.

\bibliography{bib-2}

\end{document}